\documentclass{svjour3}                     
\smartqed  
\usepackage[CJKbookmarks=true,
            bookmarksnumbered=true,
            bookmarksopen=true,
            colorlinks=true,
            citecolor=blue,
            linkcolor=blue,
            anchorcolor=green,
            urlcolor=blue
            ]{hyperref}
\usepackage{hyperref}
\usepackage{color}
\usepackage{amsmath}
\usepackage{amssymb}
\usepackage{mathbbold}
\usepackage{mathrsfs}
\usepackage[latin1]{inputenc}
\usepackage{bm}
\usepackage{epsfig,graphicx}
\usepackage{subfigure}
\usepackage{booktabs}
\usepackage{multirow}
\usepackage{dcolumn}
\usepackage{ifthen}

\usepackage[numbers]{natbib}

\newcommand{\ud}{\mathrm{d}}
\newcommand{\fv}{true}

\journalname{J Math Chem}
\begin{document}

\title{Meshless Hermite-HDMR finite difference method for high-dimensional
Dirichlet problems
}

\titlerunning{Hermite-HDMR finite difference for Dirichlet problems}        

\author{Xiaopeng Luo    \and
        Xin Xu    \and
        Herschel Rabitz 
}


\institute{X. Luo \at
              Department of Chemistry, Princeton University, Princeton, NJ 08544, USA \\
              School of Managment and Engineering, Nanjing University, Nanjing, 210008, China\\
              \email{luo\_works@163.com,xiaopeng@princeton.edu}
           \and
           X. Xu \at
              Department of Chemistry, Princeton University, Princeton, NJ 08544, USA \\
              School of Managment and Engineering, Nanjing University, Nanjing, 210008, China \\
              \email{xuxin103@163.com,xx2@princeton.edu}
           \and
           H. Rabitz \at
              Department of Chemistry, Princeton University, Princeton, NJ 08544, USA \\
              Program in Applied and Computational Mathematics, Princeton University, Princeton, NJ 08544, USA\\
              \email{hrabitz@princeton.edu}
}

\date{}

\maketitle

\begin{abstract}
 In this paper, a meshless Hermite-HDMR finite difference method is proposed to solve high-dimensional Dirichlet problems. The approach is based on the local Hermite-HDMR expansion with an additional smoothing technique. First, we introduce the HDMR decomposition combined with the multiple Hermite series to construct a class of Hermite-HDMR approximations, and the relevant error estimate is theoretically built in a class of Hermite spaces. It can not only provide high order convergence but also retain good scaling with increasing dimensions. Then the Hermite-HDMR based finite difference method is particularly proposed for solving high-dimensional Dirichlet problems. By applying a smoothing process to the Hermite-HDMR approximations, numerical stability can be guaranteed even with a small number of nodes. Numerical experiments in dimensions up to $30$ show that resulting approximations are of very high quality.
\keywords{High-dimensional Dirichlet problems \and Meshless method \and Finite difference method \and Hermite-HDMR approximation}
\end{abstract}

\section{Introduction}
\label{HH.s1}

In this work we propose an approach to numerically solve high-dimensional Dirichlet problems. Specifically, we consider the following boundary value problem for $u:\Omega\subset\mathbb{R}^d\to\mathbb{R}$:
\begin{align}\label{HH.eq.mainE}
    \left\{\begin{array}{l}
        \frac{1}{2}\triangle u(x)-\varphi(x)=0~~\textrm{for}~~x\in\Omega, \\
        u(x)\Big|_{\partial\Omega}=v(x);
    \end{array}\right.
\end{align}
where both $\varphi$ and $v$ are $\mathbb{R}^d\to\mathbb{R}$; and assume that there exists a unique and sufficiently smooth solution for the problem. When $d$ is large, these problems face a serious computational challenge because of the so-called curse of dimensionality \cite{BellmannR1961_CurseOfDimensionality,
GriebelM2004_SparseGrids,GriebelM2006B_SparseGrids}. By this, we mean the computational cost required to approximate or to recover a $d$-dimensional function with a desired accuracy scales exponentially with $d$. There were two main attempts to overcome this difficulty. One approach is to impose very strong regularity assumptions on the target function, and another way is to assume that the target function has an expected structure, such as sparsity \cite{RauhutH2007A_SparseFourierSeries,RauhutH2007A_SparseFourierSeriesII}, or low-rank \cite{MarkovskyI2008A_Low-rankApproximation}, or a low order truncated high-dimensional model representation (HDMR) form \cite{RabitzH1998A_Representations,RabitzH1999A_Representations,
LuoXP2016A_FourierHDMR}. There can be a very close connection between the regularity condition and certain structures, for example, the dominance of low order terms of an HDMR expansion can be guaranteed by imposing the mixed regularity conditions on the target function \cite{LuoXP2016A_FourierHDMR}. Although it is not clear that such strong assumptions are actually satisfied for practical problems, imposing an extra regularity condition remains a common way to reduce the computational cost for high-dimensional problems.

\ifthenelse{\equal{\fv}{true}}{
\begin{figure}[!htb]
\begin{minipage}{1.0\textwidth}
\subfigure[Classical FD for structured grid]{
\label{HH.fig.1a}
\includegraphics[width=0.5\textwidth]{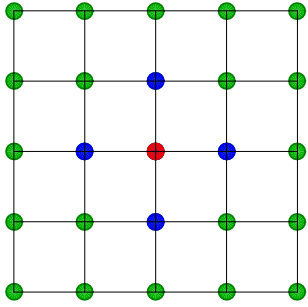}}
\subfigure[LS or RBF based FD for scattered node]{
\label{HH.fig.1b}
\includegraphics[width=0.5\textwidth]{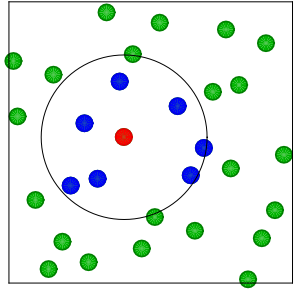}}
\end{minipage}
\caption{Existing FD methods}
\label{HH.fig.1}
\end{figure}
}

Finite difference (FD) methods are one of the simplest and the most important approaches to numerical solutions of PDEs. The traditional FD method, however, is strongly dependent on a structured grid (Fig. \ref{HH.fig.1a}), which severely limits flexibility and scalability \cite{WrightG2006M_ScatteredNodeFiniteDifference}. The FD method has been extended to a more general form (Fig. \ref{HH.fig.1b}) for scattered nodes to remove dependency on structured grids \cite{DingH2004M_LSFD,WrightG2006M_ScatteredNodeFiniteDifference}. Generally, the meshless FD method consists of approximating the derivatives of a sufficiently smooth function $f$ at a reference node (red dot in Fig. \ref{HH.fig.1b}) based on a linear combination of the values of $f$ at some surrounding grid nodes (blue dots in Fig. \ref{HH.fig.1b}), and the relevant FD weights are usually computed using polynomial interpolation on scattered nodes \cite{DingH2004M_LSFD}. In high dimensions, a very important issue is how to link a local interpolating polynomial to an imposed regularity condition.

The starting point of this work is to connect the multiple Hermite series with the HDMR decomposition and the mixed regularity condition. This approach is an extension of the work of Ref. \cite{LuoXP2016A_FourierHDMR} for Hermite polynomials. First, the mixed Hermite space $\mathcal{H}_{mix}^s(\mathbb{R}^d)$ is defined on the basis of the mixed regularity condition, and the Hermite decomposition of $\mathbb{N}_0^d$ is introduced to form an order relation for the multiple Hermite series. According to the order relation, one can truncate it to a certain order $K$ and attain a Hermite-HDMR series up to order $K$. It is mathematically proven that this truncated approximation converges very fast for functions from the mixed Hermite space $\mathcal{H}_{mix}^s(\mathbb{R}^d)$ and only has the degrees of freedom $\bm{O}(K\log^{u-1}K)$, where $u\ll d$. Then, the FD operator is generated with the help of the local weighted Hermite-HDMR expansion including an additional smoothing process.

The remainder of the paper is organized as follows. the local Hermite-HDMR approximation and the relevant error estimates are established in Section \ref{HH.s2}, and the meshless Hermite-HDMR FD method is built in Section \ref{HH.s3}. Numerical experiments for dimensions up to $30$ are given in Section \ref{HH.s5}, and finally, conclusions are presented in Section \ref{HH.s6}.

\section{Local Hermite-HDMR approximation in $\mathcal{H}_{mix}^s(\mathbb{R}^d)$}
\label{HH.s2}

The space $\mathcal{L}_2(\mathbb{R}^d)$ consists of all real-valued measurable functions on $\mathbb{R}^d$ that satisfy
\begin{equation}
    \|f\|_{\mathcal{L}_2(\mathbb{R}^d)}:=\int_{\mathbb{R}^d}
    |f(x)|^2\ud x<\infty,
\end{equation}
and then, for $s\in\mathbb{N}_0$, we define the mixed Hermite space
\begin{equation}
    \mathcal{H}_{mix}^s(\mathbb{R}^d)=\left\{f\in\mathcal{L}_2(\mathbb{R}^d):
    x_1^{s-r_1}\cdots x_d^{s-r_d}\partial^rf(x)
    \in\mathcal{L}_2(\mathbb{R}^d),~|r|_\infty\leqslant s\right\}
\end{equation}
with the norm
\begin{equation}
    \|f\|_{\mathcal{H}_{mix}^s(\mathbb{R}^d)}=
    \sum_{0\leqslant|r|_\infty\leqslant s}\left\|x_1^{s-r_1}\cdots x_d^{s-r_d}
    \partial^rf(x)\right\|_{\mathcal{L}_2(\mathbb{R}^d)},
\end{equation}
where $r=(r_1,\cdots,r_d)\in\mathbb{N}_0^d$ is a multi-index with
\begin{equation*}
    |r|_\infty=\max_{1\leqslant j\leqslant d}r_j
\end{equation*}
and
\begin{equation*}
    \partial^r=\left(\frac{\partial}{\partial x_1}\right)^{r_1}
    \cdots\left(\frac{\partial}{\partial x_d}\right)^{r_d}.
\end{equation*}
Here $\mathcal{H}_{mix}^s(\mathbb{R}^d)$ is a Banach space with respect to the norm $\|\cdot\|_{\mathcal{H}_{mix}^s(\mathbb{R}^d)}$. Further, for any $\lambda\in\mathbb{R}$ and $a\in\mathbb{R}^d$, let the space
\begin{equation}
    \mathcal{L}_2(\mathbb{R}^d,e^{-\lambda^2\|x-a\|^2}):=\left\{
    f\in\mathcal{L}_2(\mathbb{R}^d):\int_{\mathbb{R}^d}
    |f(x)|^2e^{-\lambda^2\|x-a\|^2}\ud x<\infty\right\},
\end{equation}
then
\begin{equation*}
    \mathcal{L}_2(\mathbb{R}^d,e^{-\lambda^2\|x-a\|^2})\subset
\mathcal{L}_2(\mathbb{R}^d),
\end{equation*}
and the multiply Hermite function sequence $\{H_{m,\lambda}(x-a) \}_{m\in\mathbb{N}_0^d}$ constitutes a complete orthonormal set in $\mathcal{L}_2(\mathbb{R}^d,e^{-\lambda^2\|x-a\|^2})$, where
\begin{equation}
    H_{m,\lambda}(x-a)=\sqrt{\frac{\lambda^d}{\pi^d\prod_{i=1}^d(2m_i)!!}}
    \phi_{m_1}(\lambda(x_1-a_1))
    \cdots\phi_{m_d}(\lambda(x_d-a_d))
\end{equation}
and
\begin{equation}
    \phi_j(t)=e^{t^2}\frac{\ud^j}{\ud t^j}e^{-t^2}, ~~j\in\mathbb{N}_0
\end{equation}
are the ordinary Hermite polynomials, where $(2n)!!=2\cdot4\cdot\cdots\cdot2n$; it follows that
\begin{equation}\label{HH.eq.DC}
    H_{m,\lambda}(x-a)\ud x=\frac{1}{\sqrt{(2\lambda^2)^d
    \prod_{i=1}^d(m_i+1)}}\ud H_{m+1,\lambda}(x-a),
\end{equation}
where $m+1=(m_1+1,\cdots,m_d+1)$.

Hence, for any $f\in\mathcal{L}_2(\mathbb{R}^d,e^{-\lambda^2\|x-a\|^2})$, we have the following convergent (in the sense of the norm $\|\cdot\|_{\mathcal{L}_2(\mathbb{R}^d,e^{-\lambda^2\|x-a\|^2})}$) multiple Hermite series
\begin{equation}\label{HH.eq.mHs}
    f(x)=\sum_{m\in\mathbb{N}_0^d}\alpha_m(f)H_{m,\lambda}(x-a),~~x\in\mathbb{R}^d,
\end{equation}
where
\begin{equation*}
    \alpha_m(f)=\int_{\mathbb{R}^d}
    H_{m,\lambda}(x-a)f(x)e^{-\lambda^2\|x-a\|^2}\ud x.
\end{equation*}

As a preliminary, we have the following Lemma:
\begin{lemma}\label{HH.lem.normB}
For any $\lambda\in\mathbb{R}$, $a\in\mathbb{R}^d$ and $f\in\mathcal{H}_{mix}^s(\mathbb{R}^d)$, it follows that
\begin{equation*}
    \left\|e^{\lambda^2\|x-a\|^2}\partial^{r_s}
    \left(f(x)e^{-\lambda^2\|x-a\|^2}\right)
    \right\|^2_{\mathcal{L}_2(\mathbb{R}^d)}\leqslant
    C(s,\lambda,d)\|f\|^2_{\mathcal{H}_{mix}^s(\mathbb{R}^d)},
\end{equation*}
where $r_s=(s,\cdots,s)$ is a $d$-dimensional vector and the constant $C(s,\lambda,d)$ depends on $s,\lambda$ and $d$.
\end{lemma}
\begin{proof}
 By noting that
 \begin{equation*}
    e^{\lambda^2\|x-a\|^2}\partial^{r_s}\left(f(x)e^{-\lambda^2\|x-a\|^2}\right)
 \end{equation*}
 is a linear combination of
 \begin{equation*}
    \left\{x_1^{s-r_1}\cdots x_d^{s-r_d}\partial^rf(x)
    \right\}_{0\leqslant|r|_\infty\leqslant s},
 \end{equation*}
 then the desired result holds. \qed
\end{proof}

\subsection{Hermite decomposition of $\mathbb{N}_0^d$}

In order to further discuss the multiple Hermite series \eqref{HH.eq.mHs}, we first consider the following Hermite decomposition of $\mathbb{N}_0^d$.
\begin{definition}
Suppose $d,k\in\mathbb{N}$, $\lambda>0$ and $c\geqslant1$. Let
\begin{align*}
    \Gamma^k(d,c,\lambda):=\left\{m\in\mathbb{N}_0^d:
    k\leqslant(2\lambda^2)^d\prod_{j=1}^d(m_j+c)<k+1\right\}
\end{align*}
and
\begin{equation*}
    \Gamma^K(d,c,\lambda):=\sum_{k=1}^K\Gamma^k(d,c,\lambda),
\end{equation*}
then we have the following Hermite decomposition
\begin{equation}\label{lmfs.eq.decomNd}
    \mathbb{N}_0^d=\sum_{k=1}^\infty\Gamma^k(d,c,\lambda),
\end{equation}
and $k$ is referred to as the Hermite order number of
$m\in\Gamma^k(d,c,\lambda)$.
\end{definition}
\begin{remark}
We will further discuss the constant $c$ later.
\end{remark}

Now consider an upper bound for the cardinality of $\Gamma^K$. First, by noting
\begin{align*}
    \sum_{j=2}^K\frac{1}{j}<\int_1^K\frac{1}{t}\ud t=\log K,
\end{align*}
it follows that
\begin{lemma}
Given $K\in\mathbb{N}$, we have
\begin{align*}
    \bm{h}_K<1+\log K,
\end{align*}
where $\bm{h}_K=\sum_{j=1}^Kj^{-1}$ is the $K$th harmonic number.
\end{lemma}
\begin{remark}
The asymptotic limit of $\bm{h}_K$ is $\gamma_E+\log K$ as $K\to\infty$, where $\gamma_E\approx0.57722$ is the Euler constant.
\end{remark}

Let $|S|$ be the Lebesgue measure of any given set $S$, then we have:
\begin{theorem}
Given $K\in\mathbb{N}$ and $c\geqslant1$, then
\begin{equation*}
    \left|\Gamma^K(d,c)\right|<K(1+\log K)^{d-1},
\end{equation*}
where $\Gamma^K(d,c)=\{m\in\mathbb{N}_0^d:\prod_{j=1}^d(m_j+c)\leqslant K\}$.
\end{theorem}
\begin{proof}
Let $\Lambda^K(u)=\left\{m\in\mathbb{N}^u:m_1\cdots m_u\leqslant K\right\}$, then it follows from $c\geqslant1$ that $\left|\Gamma^K(d,c)\right|< \left|\Lambda^K(d)\right|$, and it holds that $|\Lambda^K(d)|\leqslant K\bm{h}_K^{d-1}=K(1+\log K)^{d-1}$ from
\begin{align*}
    \left|\Lambda^K(1)\right|=K=K\bm{h}_K^0~~\textrm{and}~~
    \left|\Lambda^K(2)\right|=\sum_{j=1}^K\left\lfloor\frac{K}{j}\right\rfloor
    \leqslant K\sum_{j=1}^K\frac{1}{j}=K\bm{h}_K
\end{align*}
and
\begin{align*}
    \left|\Lambda^K(u+1)\right|
    \leqslant\sum_{j=1}^K\frac{\left|\Lambda^K(u)\right|}{j}
    \leqslant K\bm{h}_K^{u-1}\sum_{j=1}^K\frac{1}{j}=K\bm{h}_K^u,
    ~~\forall u\in\mathbb{N},
\end{align*}
where $\lfloor t\rfloor$ is the unique integer satisfying the inequalities
$\lfloor t\rfloor\leqslant t\leqslant\lfloor t\rfloor+1$ for any $t\in\mathbb{R}$. \qed
\end{proof}

Furthermore, we can get an upper bound for both $|\Gamma^K(d,c,\lambda)|$ and $|\Gamma^k(d,c,\lambda)|$.
\begin{corollary}\label{lmfs.cor.s2.KLbound}
Given $K\in\mathbb{N}$ and $\lambda>0$, then for any $k\leqslant K$,
\begin{equation*}
    \left|\Gamma^k(d,c,\lambda)\right|
    \leqslant\left|\Gamma^K(d,c,\lambda)\right|
    \leqslant\frac{K_\lambda+1}{(2\lambda^2)^d}\left(1+\log\left(1+
    \frac{K_\lambda+1}{(2\lambda^2)^d}\right)\right)^{d-1}.
\end{equation*}
\end{corollary}

\subsection{Multiple Hermite series in $\mathcal{H}_{mix}^s(\mathbb{R}^d)$}

According to the Hermite decomposition of $\mathbb{N}_0^d$, in the neighborhood $U(a)$ of a given point $a\in\mathbb{R}^d$, the Hermite series of $f\in\mathcal{H}_{mix}^s(\mathbb{R}^d)\subset\mathcal{L}_2(\mathbb{R}^d)$ can be rewritten as
\begin{equation}\label{HH.eq.mHsD}
    f(x)=\sum_{k=1}^\infty\sum_{m\in\Gamma^k(d,c,\lambda)}\alpha_m(f)
    H_{m,\lambda}(x-a),~~x\in U(a)\subset\mathbb{R}^d.
\end{equation}
This form is very useful for analyzing a function in the mixed Hermite space $\mathcal{H}_{mix}^s(\mathbb{R}^d)$. First, we have the following lemma.
\begin{lemma}\label{HH.lem.decay}
Suppose $s\in\mathbb{N}$ and $f\in\mathcal{H}_{mix}^s(\mathbb{R}^d)$. For any $\lambda\in\mathbb{R}$, $a\in\mathbb{R}^d$ and $m\in\Gamma^k(d,c,\lambda)$,
\begin{equation*}
    \sum_{m\in\Gamma^k(d,c,\lambda)}\alpha_m^2(f)\leqslant
    \frac{C(s,\lambda,d)\|f\|^2_{\mathcal{H}_{mix}^s(\mathbb{R}^d)}}{k^s},
\end{equation*}
where $k\leqslant(2\lambda^2)^d\prod_{j=1}^d(m_j+c)<k+1$ and the constant $c\geqslant1$ depends only on $s$.
\end{lemma}
\begin{proof}
According to integration by parts and \eqref{HH.eq.DC}, there exist $c\geqslant1$ such that
\begin{align*}
    &\sum_{m\in\Gamma^k(d,c,\lambda)}\alpha_m^2(f) \\
    =&\sum_{m\in\Gamma^k(d,c,\lambda)}\left|\int_{\mathbb{R}^d}
    H_{m,\lambda}(x-a)f(x)e^{-\lambda^2\|x-a\|^2}\ud x\right|^2 \\
    =&\sum_{m\in\Gamma^k(d,c,\lambda)}\!\!\frac{1}{(2\lambda^2)^{ds}(m)^{(s)}}\!
    \left|\int_{\mathbb{R}^d}\!\!H_{m+r_s,\lambda}(x\!-\!a)\!\!
    \left[e^{\lambda^2\|x-a\|^2}
    \partial^{r_s}\!\!\left(\!f(x)e^{-\lambda^2\|x-a\|^2}\right)\!\right]
    \!e^{-\lambda^2\|x-a\|^2}\ud x\right|^2 \\
    \leqslant&\frac{1}{k^s}\sum_{m\in\Gamma^k(d,c,\lambda)}
    \alpha_{m+r_s}^2\left(e^{\lambda^2\|x-a\|^2}
    \partial^{r_s}\!\!\left(\!f(x)e^{-\lambda^2\|x-a\|^2}\right)\right) \\
    \leqslant&\frac{\left\|e^{\lambda^2\|x-a\|^2}\partial^{r_s}
    \left(f(x)e^{-\lambda^2\|x-a\|^2}\right)
    \right\|^2_{\mathcal{L}_2(\mathbb{R}^d,e^{-\lambda^2\|x-a\|^2})}}{k^s},
\end{align*}
where $r_s=(s,\cdots,s)$ is a $d$-dimensional vector and
\begin{equation*}
    (m)^{(s)}=\prod_{j=1}^d~\big[(m_j+1)\cdots(m_j+s)\big].
\end{equation*}

From Lemma \ref{HH.lem.normB},
\begin{align*}
    &\left\|e^{\lambda^2\|x-a\|^2}\partial^{r_s}
    \left(f(x)e^{-\lambda^2\|x-a\|^2}\right)
    \right\|^2_{\mathcal{L}_2(\mathbb{R}^d,e^{-\lambda^2\|x-a\|^2})} \\
    \leqslant&\left\|e^{\lambda^2\|x-a\|^2}\partial^{r_s}\left(f(x)
    e^{-\lambda^2\|x-a\|^2}\right)\right\|^2_{\mathcal{L}_2(\mathbb{R}^d)}
    \leqslant C(s,\lambda,d)\|f\|^2_{\mathcal{H}_{mix}^s(\mathbb{R}^d)},
\end{align*}
then we have
\begin{align*}
    \sum_{m\in\Gamma^k(d,c,\lambda)}\alpha_m^2(f)\leqslant
    \frac{C(s,\lambda,d)\|f\|^2_{\mathcal{H}_{mix}^s(\mathbb{R}^d)}}{k^s},
\end{align*}
as claimed. \qed
\end{proof}

\begin{theorem}\label{HH.thm.err}
Suppose $f\in\mathcal{H}_{mix}^s(\mathbb{R}^d)$. If $s\geqslant1$, then for any $\lambda\in\mathbb{R}$ and $a\in\mathbb{R}^d$,
\begin{align*}
    &\left\|f(x)-\sum_{k=1}^K\sum_{m\in\Gamma^k(d,c,\lambda)}
    \alpha_m(f)H_{m,\lambda}(x-a)\right\|^2_{\mathcal{L}_2
    (\mathbb{R}^d,e^{-\lambda^2\|x-a\|^2})}\\
    \leqslant & ~
    \frac{C(s,\lambda,d)}{s-1}~\|f\|^2_{\mathcal{H}_{mix}^s(\mathbb{R}^d)}
    \cdot K^{-(s-1)}.
\end{align*}
\end{theorem}
\begin{proof}
In the sense of the norm $\|\cdot\|_{\mathcal{L}_2(\mathbb{R}^d,e^{-\lambda^2\|x-a\|^2})}$, the multiple Hermite series is convergent, so it follow that
\begin{align*}
    &\left\|f(x)-\sum_{k=1}^K\sum_{m\in\Gamma^k(d,c,\lambda)}
    \alpha_m(f)H_{m,\lambda}(x-a)\right\|^2_{\mathcal{L}_2
    (\mathbb{R}^d,e^{-\lambda^2\|x-a\|^2})} \\
    =&\sum_{k=K+1}^\infty\sum_{m\in\Gamma^k(d,c,\lambda)}
    \alpha_m^2(f)\|H_{m,\lambda}(x-a)\|^2_{\mathcal{L}_2
    (\mathbb{R}^d,e^{-\lambda^2\|x-a\|^2})} \\
    \leqslant&C(s,\lambda,d)\|f\|^2_{\mathcal{H}_{mix}^s(\mathbb{R}^d)}
    \sum_{k=K+1}^\infty\frac{1}{k^s} \\
    \leqslant&C(s,\lambda,d)\|f\|^2_{\mathcal{H}_{mix}^s(\mathbb{R}^d)}
    \sum_{k=K+1}^\infty\int_{k-1}^k\frac{1}{t^s}\ud t \\
    =&C(s,\lambda,d)\|f\|^2_{\mathcal{H}_{mix}^s(\mathbb{R}^d)}
    \int_K^\infty\frac{1}{t^s}\ud t \\
    =&\frac{C(s,\lambda,d)}{s-1}\|f\|^2_{\mathcal{H}_{mix}^s(\mathbb{R}^d)}
    \cdot K^{-(s-1)},
\end{align*}
and the proof is complete. \qed
\end{proof}

\begin{corollary}\label{HH.cor.err}
Suppose $f\in\mathcal{H}_{mix}^s(\mathbb{R}^d)$ and $r\in\mathbb{N}_0^d$. If $s\geqslant|r|_\infty+1$, then for any $\lambda\in\mathbb{R}$ and $a\in\mathbb{R}^d$,
\begin{align*}
    &\left\|\partial^rf(x)-\sum_{k=1}^K
    \sum_{m\in\Gamma^k(d,c,\lambda)}\!\!\!\!\!\alpha_m(f)
    \partial^rH_{m,\lambda}(x-a)
    \right\|^2_{\mathcal{L}_2(\mathbb{R}^d,e^{-\lambda^2\|x-a\|^2})}\\
    \leqslant &~\frac{C(s,\lambda,d)}{s\!-\!|r|_\infty\!-\!1}
    \|f\|^2_{\mathcal{H}_{mix}^s(\mathbb{R}^d)}
    \cdot K^{-(s-|r|_\infty-1)}.
\end{align*}
\end{corollary}
\begin{proof}
According to \eqref{HH.eq.DC}, there exist $c\geqslant1$ such that
\begin{align*}
    \|\partial^rH_{m,\lambda}(x-a)\|^2_{\mathcal{L}_2
    (\mathbb{R}^d,e^{-\lambda^2\|x-a\|^2})}
    =&(2\lambda^2)^{d|r|}(m)_{(r)}\leqslant k^{|r|_\infty},
\end{align*}
where
\begin{equation*}
    (m)_{(r)}=\prod_{j=1}^dm_j(m_j-1)\cdots(m_j-r_j);
\end{equation*}
then
\begin{equation*}
    \left\|\sum_{m\in\Gamma^k(d,c,\lambda)}\alpha_m(f)
    \partial^rH_{m,\lambda}(x-a)\right\|^2
    _{\mathcal{L}_2(\mathbb{R}^d,e^{-\lambda^2\|x-a\|^2})}
    \leqslant\frac{C(s,\lambda,d)\|f\|^2_{\mathcal{H}_{mix}^s(\mathbb{R}^d)}}
    {k^{s-|r|_\infty}}.
\end{equation*}
Hence, from the proof of Theorem \ref{HH.thm.err}, the desired result holds. \qed
\end{proof}

\subsection{The fundamental conjecture of HDMR in $\mathcal{H}_{mix}^s(\mathbb{R}^d)$}

\begin{definition}
Given a nonempty set $\{j_1,\cdots,j_u\}\subset\{1,\cdots,d\}$, then
\begin{equation*}
    \Gamma_{j_1,\cdots,j_u}^k(d,c,\lambda)
    :=\{m\in\Gamma^k(d,c,\lambda): m_l>0~\forall
    l\in\{j_1,\cdots,j_u\}~\textrm{and}~m_l=0
    ~\forall l\notin\{j_1,\cdots,j_u\}\},
\end{equation*}
where $l\in\{1,\cdots,d\}$.
\end{definition}

Using these sets, $\Gamma^k(d,c,\lambda)$ can be decomposed into the following form
\begin{align}\label{HH.eq.decomZd}
    \Gamma^k(d,c,\lambda)=\Gamma_0^k(d,c,\lambda)+\!\!\sum_{j}\Gamma_j^k(d,c,\lambda)
    +\!\!\!\sum_{j_1<j_2}\!\!\Gamma_{j_1,j_2}^k(d,c,\lambda)+\cdots+
    \Gamma_{j_1\!,\cdots\!,j_d}^k(d,c,\lambda)
\end{align}
and we refer to this as a HDMR decomposition of $\Gamma^k(d,c,\lambda)$; then for any $f\in\mathcal{L}_2(\mathbb{R}^d,e^{-\lambda^2\|x-a\|^2})$, the multiple Hermite series \eqref{HH.eq.mHsD} can be rewritten as
\begin{equation}\label{HH.eq.mHHsD}
\begin{split}
    f(x)=\sum_{k=1}^\infty&\left(\sum_{m\in\Gamma_0^k(d,c,\lambda)}\alpha_m(f)
    H_{m,\lambda}(x-a)+\sum_{j}\sum_{m\in\Gamma_j^k(d,c,\lambda)}
    \alpha_m(f)H_{m,\lambda}(x-a)
    \right. \\
    &\left.+\!\!\!\sum_{j_1<j_2}\sum_{m\in\Gamma_{j_1,j_2}^k(d,c,\lambda)}
    \!\!\!\!\!\!\!\alpha_m(f)H_{m,\lambda}(x-a)\!+\cdots \! +
    \!\!\!\!\!\!\!\sum_{m\in\Gamma_{j_1\!,\cdots \!,j_d}^k(d,c,\lambda)}
    \!\!\!\!\!\!\!\!\alpha_m(f)H_{m,\lambda}(x-a)\!\right),
\end{split}
\end{equation}
which is referred to this as a Hermite-HDMR decomposition of $f$ from $\mathcal{L}_2(\mathbb{R}^d,e^{-\lambda^2\|x-a\|^2})$.
\begin{theorem}\label{HH.thm.HDMR}
Suppose $f\in\mathcal{H}_{mix}^s(\mathbb{R}^d)$, $r\in\mathbb{N}_0^d$ and $s\geqslant|r|_\infty+1$. For a fixed $K\in\mathbb{N}$, if $u=u^*$ is the smallest integer that satisfies
\begin{equation*}
    (1+c)^{u+1}>\frac{K+1}{(2\lambda^2)^d}, ~~1\leqslant u\leqslant d,
\end{equation*}
then
\begin{equation}\label{HH.eq.HDMR}
    \|\partial^rf(x)-\partial^r\mathcal{S}_{u^*}f(x)\|^2
    _{\mathcal{L}_2(\mathbb{R}^d,e^{-\lambda^2\|x-a\|^2})}
    \leqslant\frac{C(s,\lambda,d)}{s-|r|_\infty-1}
    \|f\|^2_{\mathcal{H}_{mix}^s(\mathbb{R}^d,
    \lambda)}\cdot K^{-(s-|r|_\infty-1)}.
\end{equation}
\end{theorem}
\begin{proof}
Denote
\begin{equation*}
    \big\langle p,q\big\rangle=\int_{\mathbb{R}^d}p(x)q(x)
    e^{-\lambda^2\|x-a\|^2}\ud x,
\end{equation*}
where $p,q\in\mathcal{L}_2(\mathbb{R}^d,e^{-\lambda^2\|x-a\|^2})$; let
\begin{align*}
    A(x)=&\sum_{k=1}^K\sum_{m\in\Gamma^k(d,c,\lambda)}
    \alpha_m(f)H_{m,\lambda}(x-a),\\
    B(x)=&\sum_{k=K+1}^\infty\sum_{m\in\Gamma^k(d,c,\lambda)}
    \alpha_m(f)H_{m,\lambda}(x-a),
\end{align*}
that is,
\begin{equation*}
    f(x)=A(x)+B(x)
\end{equation*}
and
\begin{equation*}
    \big\langle f,A\big\rangle=\big\langle A,A\big\rangle,~~
    \big\langle f,B\big\rangle=\big\langle B,B\big\rangle.
\end{equation*}

Noting that
\begin{equation*}
    \big\langle \mathcal{S}_{u^*}f,f-\mathcal{S}_{u^*}f\big\rangle=0
\end{equation*}
and
\begin{align*}
    \big\langle A,f-\mathcal{S}_{u^*}f\big\rangle=0~~~\textrm{when}~~~
    (1+c)^{u^*}>\frac{k+1}{(2\lambda^2)^d},
\end{align*}
we have
\begin{align*}
    \big\langle f-\mathcal{S}_{u^*}f,f-\mathcal{S}_{u^*}f\big\rangle
    =\big\langle A+B,f-\mathcal{S}_{u^*}f\big\rangle
    =\big\langle B,f\big\rangle-\big\langle B,\mathcal{S}_{u^*}f\big\rangle
    =\big\langle B,B\big\rangle-\big\langle B,\mathcal{S}_{u^*}f\big\rangle,
\end{align*}
since $\big\langle B,\mathcal{S}_{u^*}f\big\rangle\geqslant0$,
\begin{equation*}
    \|f-\mathcal{S}_{u^*}f\|_{\mathcal{L}_2(\mathbb{R}^d,e^{-\lambda^2\|x-a\|^2})}
    \leqslant\|B\|_{\mathcal{L}_2(\mathbb{R}^d,e^{-\lambda^2\|x-a\|^2})},
\end{equation*}
and then \eqref{HH.eq.HDMR} holds. \qed
\end{proof}

This proof also reveals that the Hermite-HDMR approximation up to order $K$, that is,
\begin{equation}\label{HH.eq.app}
    \sum_{k=1}^K\sum_{m\in\Gamma^k(d,c,\lambda)}\alpha_m(f)H_{m,\lambda}(x-a)
\end{equation}
is a truncated HDMR expansion of not more than order $u$ if
\begin{equation*}
    (1+c)^{u+1}>\frac{K+1}{(2\lambda^2)^d}, ~~1\leqslant u\leqslant d.
\end{equation*}
Hence, the degrees of freedom of \eqref{HH.eq.app} is
\begin{equation*}
    \bm{O}\left(\frac{K+1}{(2\lambda^2)^d}\left(1+\log\left(1+
    \frac{K+1}{(2\lambda^2)^d}\right)\right)^{u-1}\right),
\end{equation*}
which is the main reason that the Hermite-HDMR approximation can be used for high dimensional problems.

\section{The meshless Hermite-HDMR FD method}
\label{HH.s3}

\subsection{Hermite-HDMR FD method}

Let $\boldsymbol{\chi}=\{\chi_i\}_{i=1}^N\subset\Omega$ be the interior node set with the size $N$ and $\boldsymbol{\chi'}=\{\chi'_i\}_{i=1}^{N'}\subset\Omega$ be the total node set with size $N'$ (including the boundary nodes). According to
\begin{align}\label{HH.eq.rG}
    \Gamma^K(d,c):=\left\{m\in\mathbb{N}_0^d:
    \prod_{j=1}^d(m_j+c)<K\right\},
\end{align}
for functions from $\mathcal{H}_{mix}^s(\mathbb{R}^d)$, we redefine the Hermite-HDMR smoothing up to order $K$ at a neighborhood of $\chi_i$, i.e.,
\begin{equation*}
    S_K(x,\chi_i,\beta)=\sum_{m\in\Gamma^K(d,c)}\frac{\alpha_m}{k_m^\beta}
    H_{m,\lambda}(x-\chi_i),
\end{equation*}
where $\beta\geqslant0$ is the smoothing factor and
\begin{equation*}
    k_m=\prod_{j=1}^d(m_j+c).
\end{equation*}
Let $\{H^{(j)}(x)\}_{j=1}^M$ denote the bases $\{H_{m,\lambda}( x)\}_{m\in\Gamma^K(d,c)}$ and $\{k_{(j)}\}_{j=1}^M$ denote $\{k_m\}_{m\in\Gamma^K(d,c)}$, i.e., $M=|\Gamma^K(d,c)|$, then for a given reference node $\chi_i\in\boldsymbol{\chi}$, the solution $u(x)$ of the problem \eqref{HH.eq.mainE} can be approximately represented as
\begin{equation}\label{HH.eq.LocalA}
    S_M(x,\chi_i,\beta)=\sum_{j=1}^M\frac{H^{(j)}(x-\chi_i)}
    {k_{(j)}^\beta}\alpha^{(j)},
\end{equation}
then the interpolation equations at the reference node $\chi_i$ can be written as
\begin{align*}
    \left(
      \begin{array}{cccc}
        \frac{H^{(1)}(\chi_1-\chi_i)}{k_{(1)}^\beta} & \cdots & \frac{H^{(M)}(\chi_1-\chi_i)}{k_{(M)}^\beta} \\
        \vdots & \cdots & \vdots \\
        \frac{H^{(1)}(\chi_{N'}-\chi_i)}{k_{(1)}^\beta} & \cdots & \frac{H^{(M)}(\chi_{N'}-\chi_i)}{k_{(M)}^\beta} \\
      \end{array}
    \right)
    \left(
      \begin{array}{c}
        \alpha_1 \\
        \vdots \\
        \alpha_M \\
      \end{array}
    \right)
    =\left(
       \begin{array}{c}
         u(\chi_1) \\
         \vdots \\
         u(\chi_{N'}) \\
       \end{array}
     \right),
\end{align*}
which in matrix notation becomes
\begin{equation}\label{HH.eq.IE}
    \bm{H}_i\bm{\alpha}_i=\bm{U'}.
\end{equation}
Suppose $\bm{\hat{\alpha}}_i=\bm{C}_i\bm{U'}$ is the least square solution of \eqref{HH.eq.IE} under the weighted $2$-norm
\begin{equation*}
    \|\bm{\alpha}\|_{2,w}=\bm{\alpha}^\mathrm{T}\bm{W}\bm{\alpha},
\end{equation*}
where $\bm{W}$ is the diagonal matrix constructed from the constant weights, i.e.,
\begin{equation*}
    \bm{W}=\textrm{diag}\left(e^{-\lambda^2\|\chi_1\|_2^2},\cdots,
    e^{-\lambda^2\|\chi_{N'}\|_2^2}\right).
\end{equation*}
Thus, at the reference node $\chi_i$, the Laplace operator $\triangle u(\chi_i)$ can be approximated by
\begin{equation*}
    \sum_{j=1}^M\frac{\triangle H^{(j)}(0)}{k_{(j)}^\beta}\cdot\alpha_i^{(j)}
    =\bm{D}_i^\mathrm{T}\bm{\alpha}_i=\bm{D}_i^\mathrm{T}\bm{C}_i\bm{U'}
\end{equation*}
and on the basis of the interior node set $\boldsymbol{\chi}= \{\chi_i\}_{i=1}^N\subset\Omega$, the equation $\frac{1}{2}\triangle u(x)=\varphi(x)$ is discretized as
\begin{equation*}
    \left(
      \begin{array}{c}
        \bm{D}_1^\mathrm{T}\bm{C}_1 \\
        \vdots \\
        \bm{D}_N^\mathrm{T}\bm{C}_N \\
      \end{array}
    \right)\bm{U'}=
    \left(
      \begin{array}{c}
        \varphi(\chi_1) \\
        \vdots \\
        \varphi(\chi_N) \\
      \end{array}
    \right),
\end{equation*}
or in matrix notation,
\begin{equation*}
    \bm{A'}\bm{U'}=\Phi.
\end{equation*}
By imposing the boundary conditions ($u(\chi_s)=v(\chi_s)$ is known if $\chi_s\in\boldsymbol{\chi'}-\boldsymbol{\chi}$), we have the final difference system
\begin{equation}\label{HH.eq.DE}
    \bm{A}\bm{U}=\Phi'.
\end{equation}
Due to the fast decay of the Gaussian weighted function, the difference matrix $\bm{A}$ is very sparse, so this difference system can be solved efficiently by iteration methods, such as the biconjugate gradient stabilized method (BCGS) \cite{VanderVorstHA1992A_BI-CGSTAB,GutknechtMH1993A_BICGSTAB} or the successive over-relaxation (SOR) method.

\subsection{The behavior of the Hermite-HDMR FD approximation}

Since the fast decay of the weighted function $e^{-\lambda^2\|x-\chi_i\|^2}$, those nodes located outside the $d$-ball $B(\chi_i,\rho)=\{x\in\mathbb{R}^d: \|x-\chi_i\|_2\leqslant\rho\}$ contribute very little to the reference node $\chi_i$, where the radius $\rho$ is inversely related to $\lambda$, i.e.,
\begin{equation*}
    \rho=\frac{\kappa}{\lambda},
\end{equation*}
and throughout this paper, we use $\kappa=2.628$. According to \eqref{HH.eq.LocalA}, the number of the nodes located in the $d$-ball $B(\chi_i,\rho)$ should be at least equal to $M$, hence let
\begin{equation*}
    \frac{\pi^{\frac{d}{2}}}{\Gamma(\frac{d}{2}+1)}
    \left(\frac{\kappa}{\lambda}\right)^d
    =\frac{\theta M}{N}|\Omega|,
\end{equation*}
where the constant $\theta>1$, $\Gamma(n)$ is the gamma function (i.e., $\Gamma(n)=(n-1)!$ when $n$ is a positive integer), $|\Omega|$ is the Lebesgue measure of $\Omega$, and $N$ is the size of $\boldsymbol{\chi}$; then
\begin{equation}\label{HH.eq.L}
    \left(\frac{1}{\lambda}\right)^d=\left(\frac{1}{\kappa\sqrt{\pi}}\right)^d
    \frac{\theta M\Gamma(\frac{d}{2}+1)|\Omega|}{N},
\end{equation}
and together with Corollary \ref{HH.cor.err} we have the following result:
\begin{theorem}\label{HH.thm.NLerr}
Under the conditions of \eqref{HH.eq.LocalA}, if $u(x)\in \mathcal{H}_{mix}^s(\mathbb{R}^d)$ and $s\geqslant3$, then
\begin{align*}
    &\lim_{\beta\to0}\left\|\triangle u(x)-\triangle S_M(x,\chi_i,\beta)
    \right\|_{\mathcal{L}_2(\mathbb{R}^d,e^{-\lambda^2\|x-\chi_i\|^2})}\\
    \leqslant &~\sqrt{\frac{C(s,\lambda,d)}{s\!-\!3}}~
    \|u\|_{\mathcal{H}_{mix}^s(\mathbb{R}^d)}
    \left[C''\frac{M}{N\sqrt{K}}\right]^{s-3},
\end{align*}
where $S_M(x,\chi_i,\beta)$ is defined by \eqref{HH.eq.LocalA} and
\begin{equation*}
    M=|\Gamma^K(d,c)|~~\textrm{and}~~C''=
    \theta\left(\frac{1}{\kappa\sqrt{2\pi}}\right)^d
    \Gamma\left(\frac{d}{2}+1\right)|\Omega|.
\end{equation*}
\end{theorem}

\section{Numerical examples}
\label{HH.s5}

For the given interior node set $\boldsymbol{\chi}=\{\chi_i\}_{i=1}^N$, let's define the average relative error percentage (AREP) as follows
\begin{equation}\label{HH.eq.mRE}
    \textrm{AREP}(U)=\frac{1}{N}\sum_{i=1}^N
    \left|\frac{U_i-u(\chi_i)}{u(\chi_i)}\right|\cdot100\%,
    ~~~\textrm{where}~~~u(\chi_i)\neq0.
\end{equation}
In the following, we use boxplots to express all the AREPs; and every boxplot will be generated by independently repeating the computation $10$ times with different random node sets.

\subsection{Case $1$: constant inhomogeneity, linear boundary condition}

We consider
\begin{align}\label{HH.eq.NE1}
    \varphi(x)=-1~~\textrm{and}~~v(x)=\sum_{i=1}^{30}x_i,
\end{align}
where $x\in\Omega=\{x\in\mathbb{R}^{30}:\|x\|\leqslant1\}$ and the explicit solution is
\begin{equation*}
    u(x)=\frac{1}{30}(1-\|x\|^2)+\sum_{i=1}^{30}x_i.
\end{equation*}
This problem is defined in a $30$-dimensional unit sphere, the inhomogeneity $\varphi$ is a constant and the boundary condition does not vary significantly and is linear.

We show the AREPs obtained versus node size $N$ in Fig. \ref{HH.fig.ne1}. Employing smoothing clearly improves the accuracy of the results.

\ifthenelse{\equal{\fv}{true}}{
\begin{figure}[!htb]
\centering
\begin{minipage}{1.0\textwidth}
\subfigure{\label{HH.fig.1a1}
\includegraphics[width=0.5\textwidth]{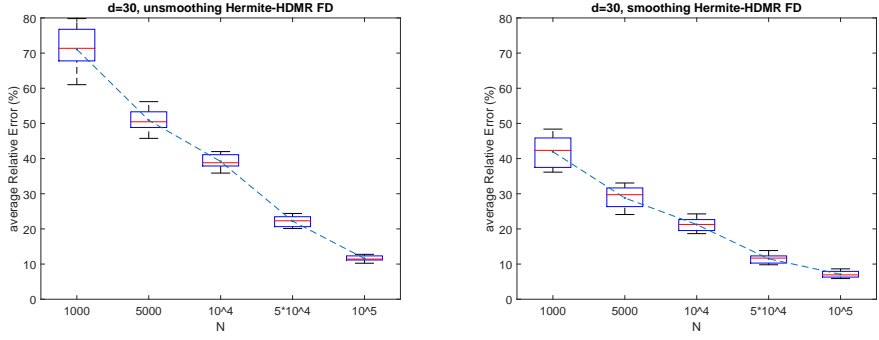}}
\subfigure{\label{HH.fig.1a2}
\includegraphics[width=0.5\textwidth]{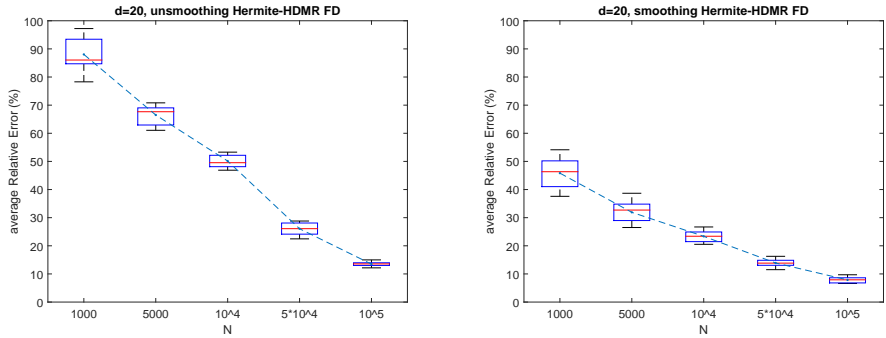}}
\end{minipage}
\caption{Results for test problem \eqref{HH.eq.NE1}}
\label{HH.fig.ne1}
\end{figure}
}

\subsection{Case $2$: quadratic inhomogeneity, quartic boundary condition}

We consider
\begin{align}\label{HH.eq.NE2}
    \varphi(x)=-\sum_{i=1}^{20}x_i^2~~\textrm{and}~~
    v(x)=\frac{1}{6}\sum_{i=1}^{20}x_i^4,
\end{align}
where $\Omega=[-1,1]^{20}$ and the explicit solution $u(x)=v(x)$. Here we make the inhomogeneity $\varphi$ more complex and choose a nonlinear boundary condition, but this problem still has an intrinsic symmetry: $\varphi$ is constant on the sphere $\|x\|=r=\textrm{const}$, and $v$ is constant when $\sum_{i=1}^{20}x_i^4=R=const$, respectively.

We show the AREPs obtained versus node size $N$ in Fig. \ref{HH.fig.ne2}.

\ifthenelse{\equal{\fv}{true}}{
\begin{figure}[!htb]
\centering
\begin{minipage}{1.0\textwidth}
\subfigure{\label{HH.fig.1a1}
\includegraphics[width=0.5\textwidth]{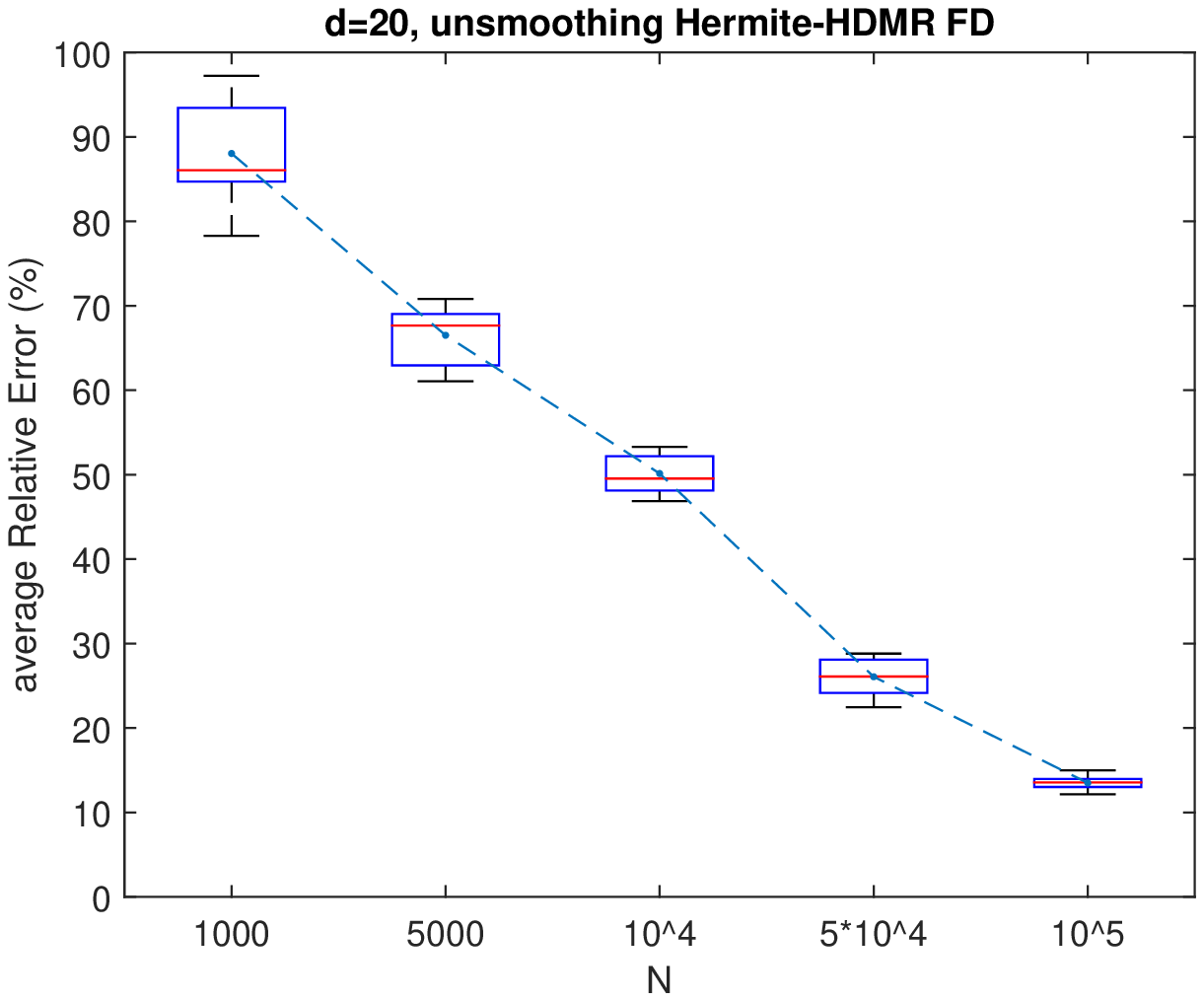}}
\subfigure{\label{HH.fig.1a2}
\includegraphics[width=0.5\textwidth]{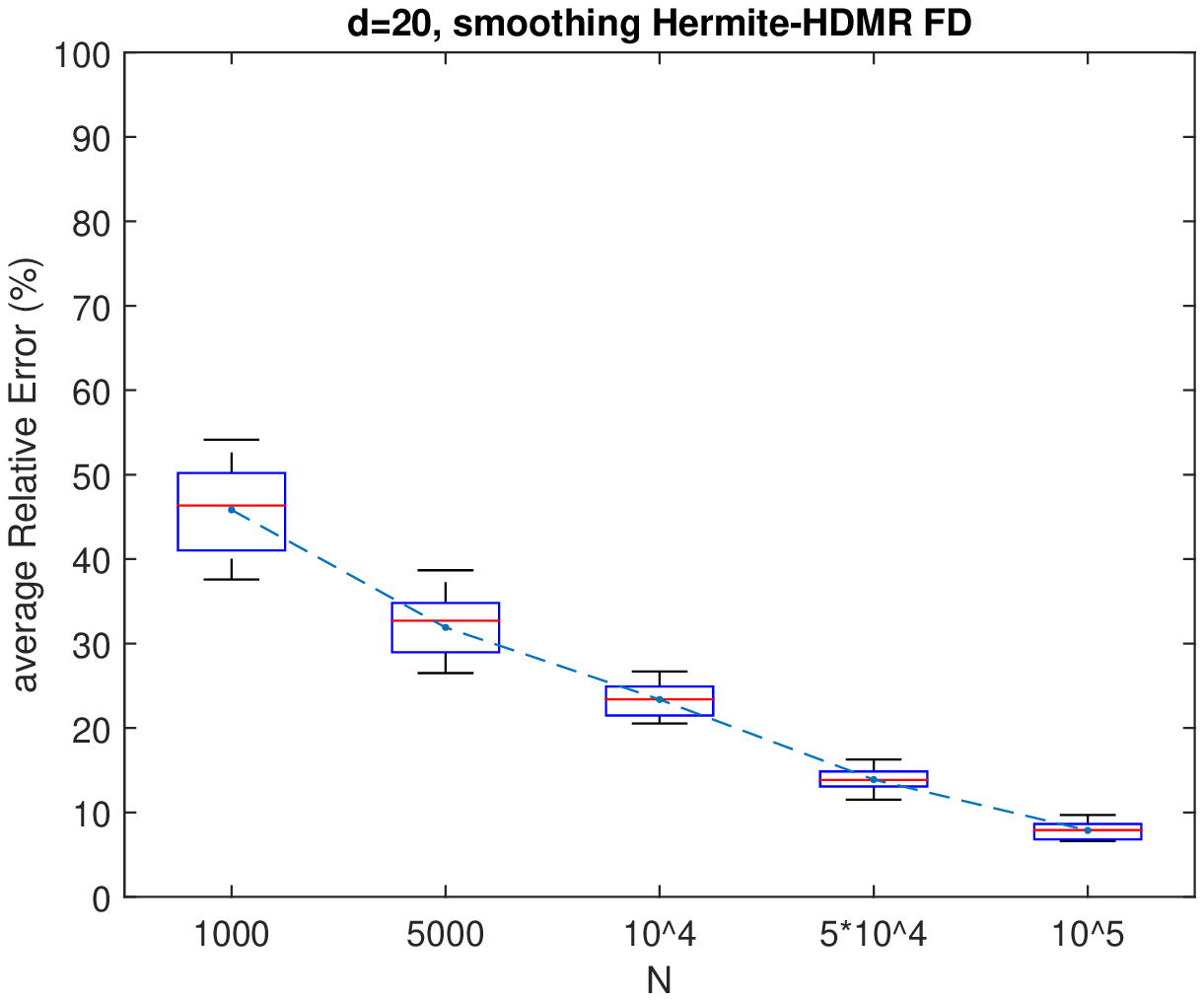}}
\end{minipage}
\caption{Results for test problem \eqref{HH.eq.NE1}}
\label{HH.fig.ne2}
\end{figure}
}

\subsection{Case $3$: transcendental inhomogeneity and boundary condition}

We consider
\begin{equation}\label{HH.eq.NE3}
\begin{split}
    \varphi(x)=&(2\|x\|^2-d)\exp(-\|x\|^2)
    -\frac{d(x_1+\cdots+x_d)}{4(4+(x_1+\cdots+x_d)^2)^2} \\
    v(x)=&\arctan\left(\frac{x_1+\cdots+x_d}{2}\right)+\exp(-\|x\|^2),
\end{split}
\end{equation}
where $\Omega=[-3,3]^d$, then the explicit solution $u(x)=v(x)$.

We show the AREPs obtained versus node size $N$ in Fig. \ref{HH.fig.ne3}.

\ifthenelse{\equal{\fv}{true}}{
\begin{figure}[!htb]
\centering
\begin{minipage}{1.0\textwidth}
\subfigure{\label{HH.fig.3a1}
\includegraphics[width=0.5\textwidth]{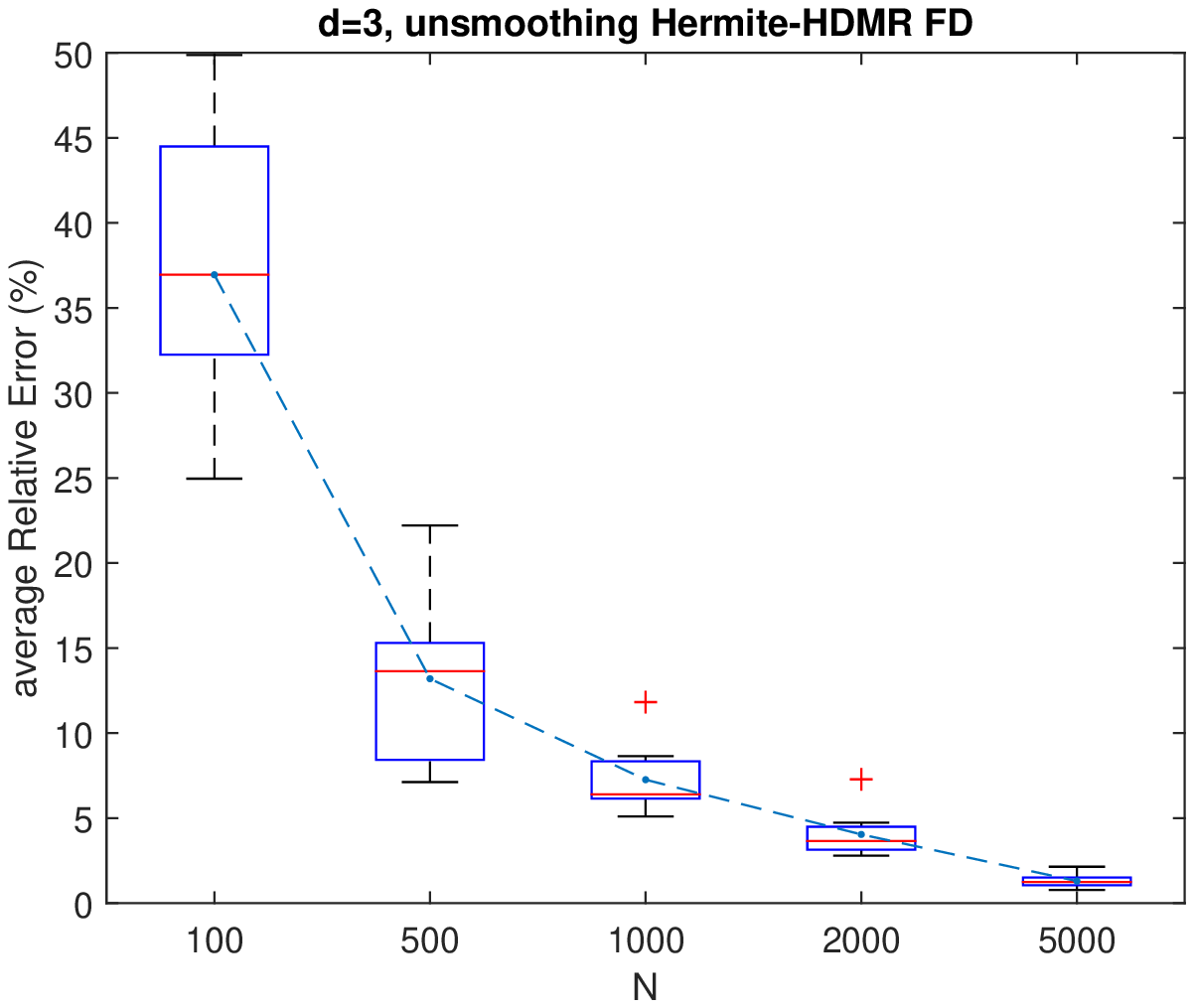}}
\subfigure{\label{HH.fig.3a2}
\includegraphics[width=0.5\textwidth]{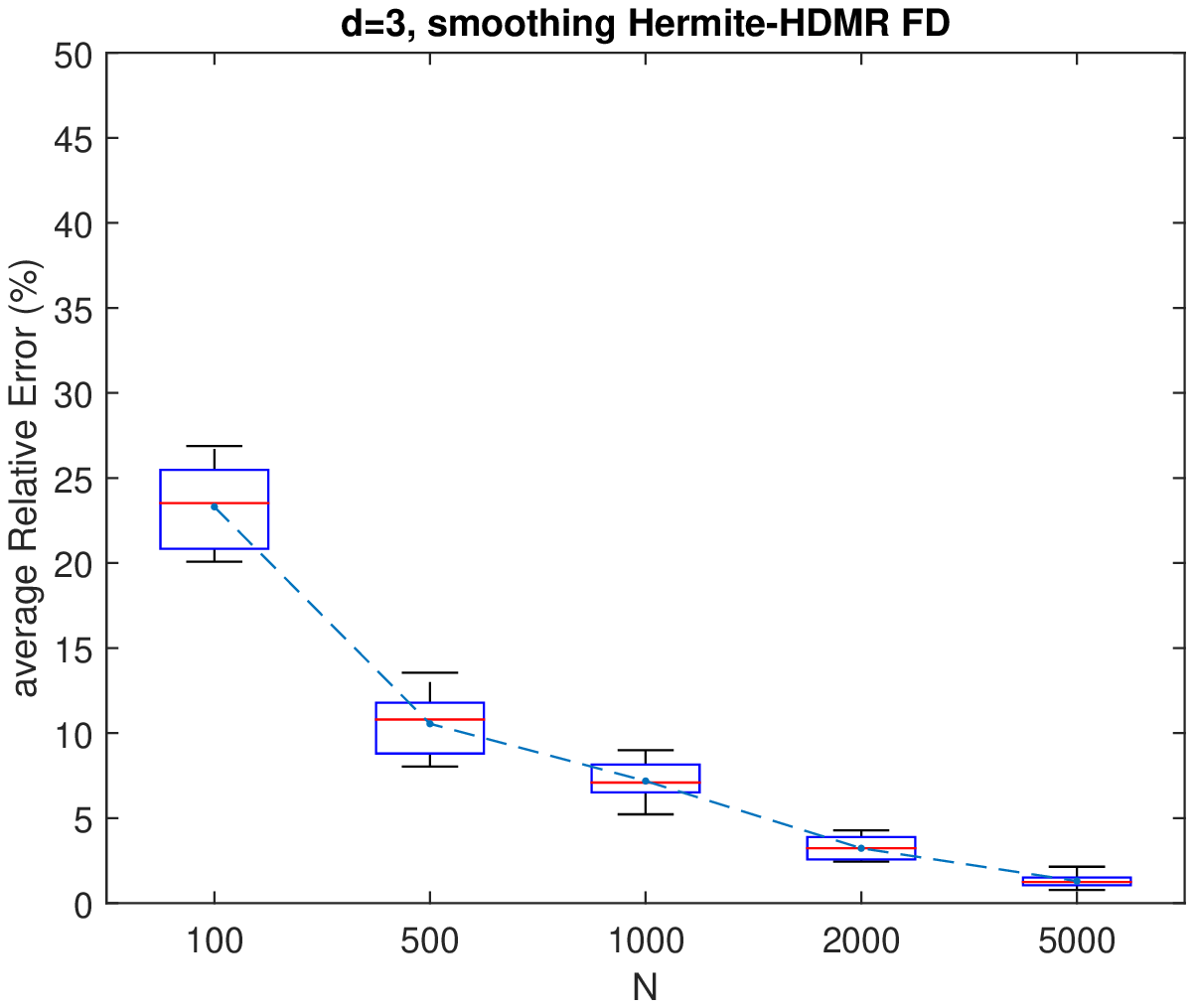}}
\end{minipage}
\begin{minipage}{1.0\textwidth}
\subfigure{\label{HH.fig.3b1}
\includegraphics[width=0.5\textwidth]{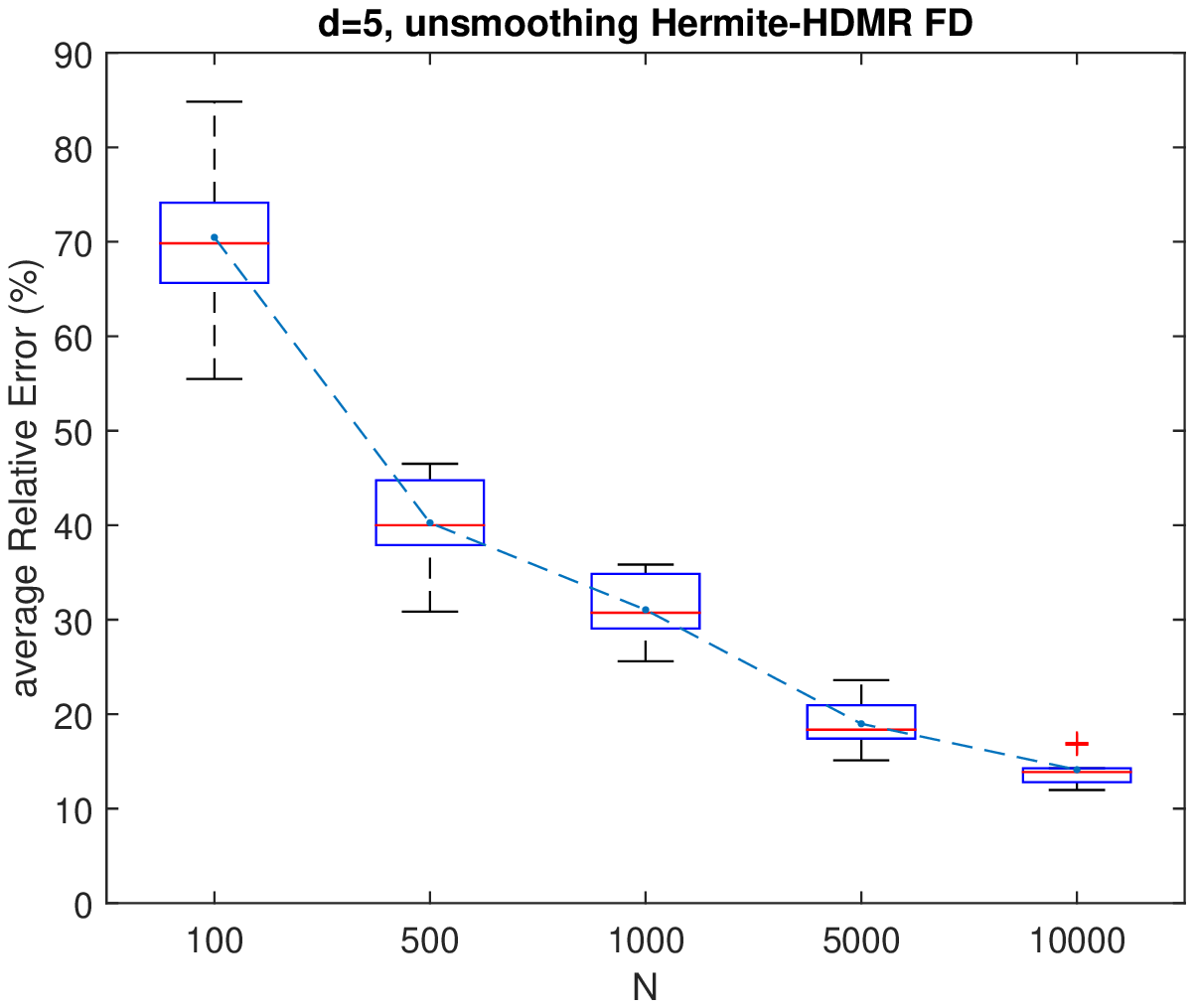}}
\subfigure{\label{HH.fig.3b2}
\includegraphics[width=0.5\textwidth]{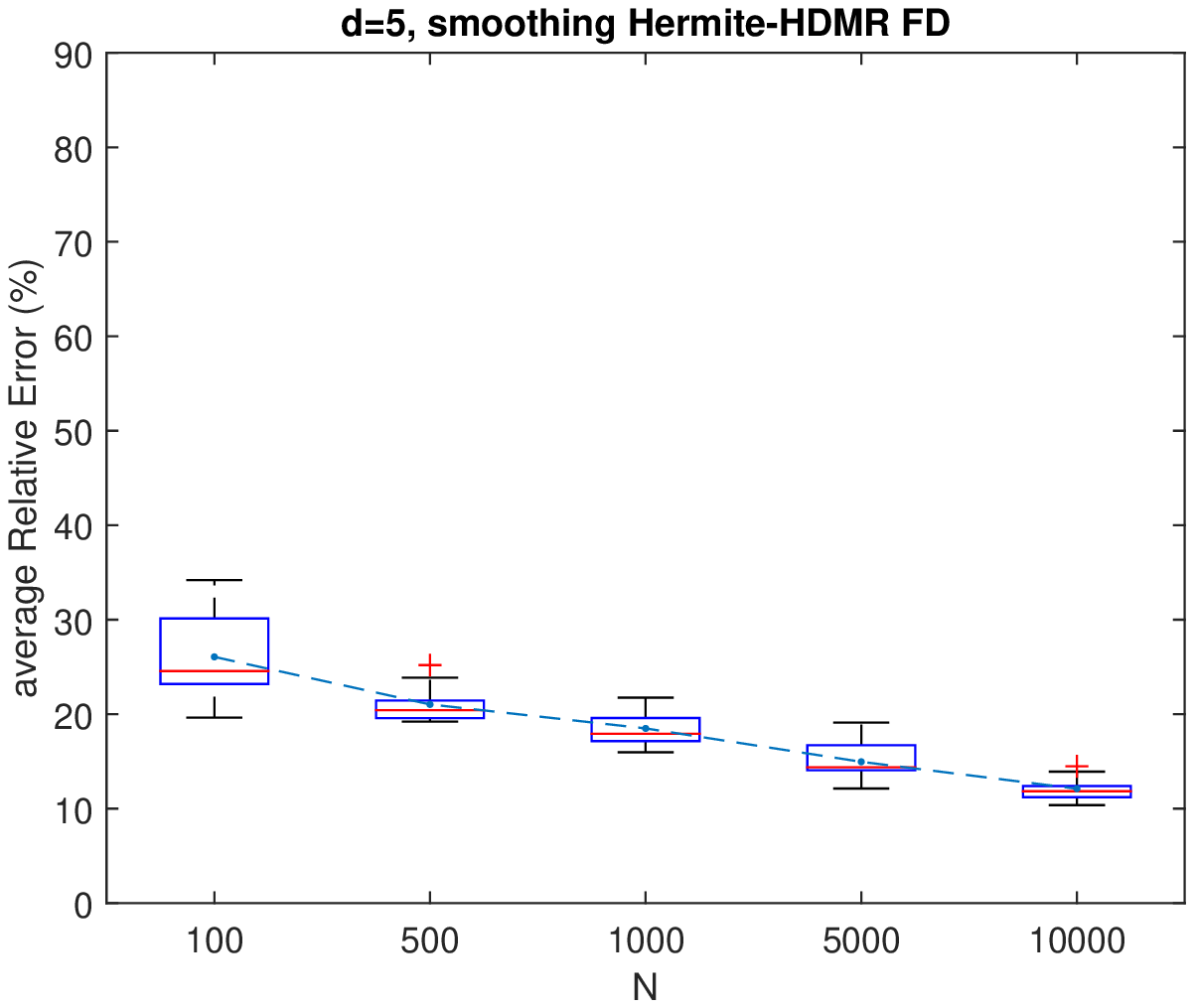}}
\end{minipage}
\begin{minipage}{1.0\textwidth}
\subfigure{\label{HH.fig.3c1}
\includegraphics[width=0.5\textwidth]{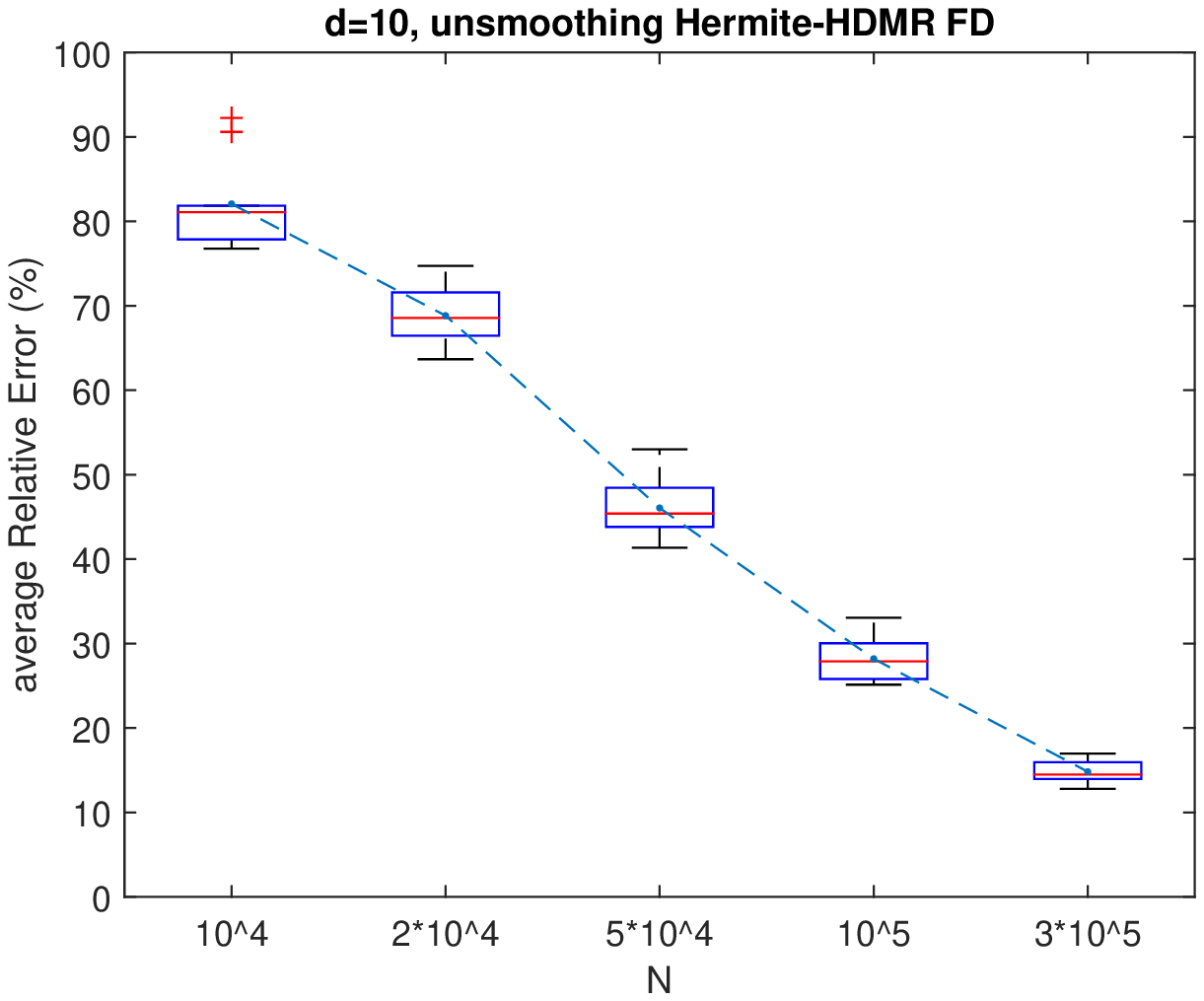}}
\subfigure{\label{HH.fig.3c2}
\includegraphics[width=0.5\textwidth]{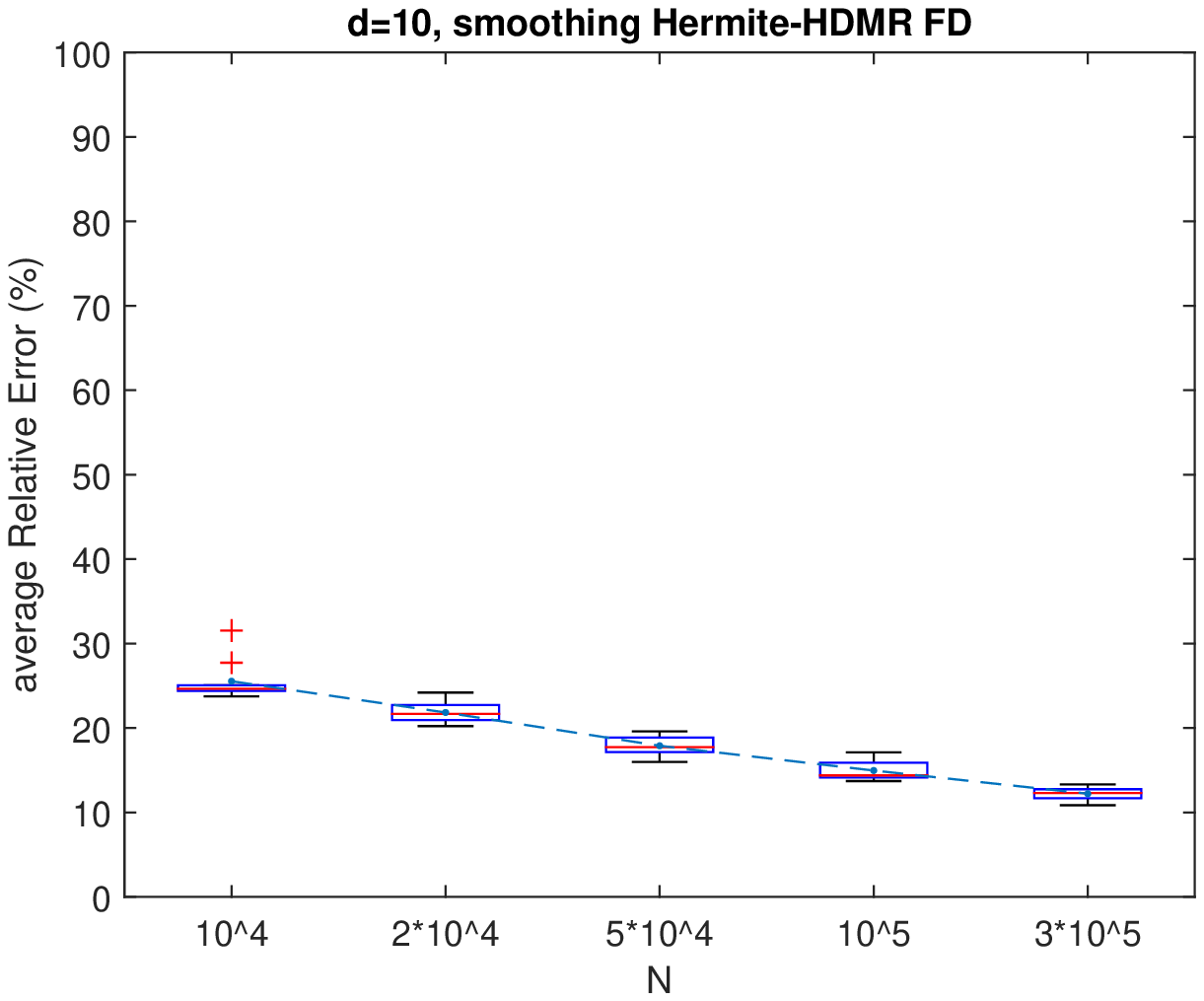}}
\end{minipage}
\caption{Results for test problem \eqref{HH.eq.NE3}}
\label{HH.fig.ne3}
\end{figure}
}

\section{Conclusions}
\label{HH.s6}

In this work, we proposed a meshless Hermite-HDMR FD method to solve high-dimensional Dirichlet problems. The approach is based on the local Hermite-HDMR expansion with an additional smoothing technique. The multiple Hermite series is connected with the HDMR decomposition and the mixed regularity condition for obtaining a class of Hermite-HDMR approximations; the relevant error estimate is theoretically built in a class of Hermite spaces. The method can not only provide high order convergence but also effectively control the degrees of freedom in high-dimensions. Numerical experiments in dimensions up to $30$ show that the resulting approximations are of very high quality, and we propose that the Hermite-HDMR finite difference method is attractive for solving high-dimensional Dirichlet problems.

\begin{acknowledgements}
  X.L. and X.X. acknowledge support from the National Science Foundation (Grant No. CHE-1763198), and H.R. acknowledges support from the Templeton Foundation (Grant No. 52265).
\end{acknowledgements}

\bibliographystyle{spbasic}
\bibliography{MReferences}

\end{document}